\documentclass[11pt]{amsart}
\usepackage{hyperref}
\usepackage{amsfonts}
\usepackage{amsmath}
\usepackage{amssymb}
\usepackage{amscd}
\usepackage{graphicx}
\usepackage{enumitem}
\usepackage{latexsym}
\usepackage{color}
\usepackage{comment}
\usepackage{epsfig}
\usepackage{amsopn}
\usepackage{mathrsfs}
\usepackage{array}
\usepackage[usenames,dvipsnames]{pstricks}
\usepackage{epsfig}
\usepackage{pst-grad} 
\usepackage{pst-plot} 
\usepackage[space]{grffile} 
\usepackage{etoolbox} 
\makeatletter 
\patchcmd\Gread@eps{\@inputcheck#1 }{\@inputcheck"#1"\relax}{}{}
\makeatother

\hypersetup{pdfpagemode=UseNone}
\setenumerate[0]{label=(\alph*)}

\usepackage{booktabs}
\usepackage{longtable}
\theoremstyle{plain}
\newtheorem{lemma}{Lemma}[section]
\newtheorem*{theorem*}{Theorem}
\newtheorem*{lemma*}{Lemma}
\newtheorem*{proposition*}{Proposition}
\newtheorem*{conjecture*}{Conjecture}
\newtheorem*{corollary*}{Corollary}
\newtheorem*{problem*}{Problem}
\newtheorem{theorem}[lemma]{Theorem}
\newtheorem{conjecture}[lemma]{Conjecture}
\newtheorem{corollary}[lemma]{Corollary}
\newtheorem{proposition}[lemma]{Proposition}

\newtheorem{question}[lemma]{Question}

\theoremstyle{definition}
\newtheorem{definition}[lemma]{Definition}
\newtheorem{example}[lemma]{Example}
\newtheorem{remark}[lemma]{Remark}

\newcommand{\fto}[1]{\stackrel{#1}{\to}}

\newcommand{\FF}{\mathbb{F}}

\newcommand{\CC}{\mathbb{C}}
\newcommand{\QQ}{\mathbb{Q}}

\newcommand{\OO}{\mathcal{O}}
\newcommand{\te}{\otimes}

\newcommand{\bw}{{\bf w}}

\newcommand{\ra}{\rightarrow}
\newcommand{\ora}{\overrightarrow}

\renewcommand{\P}{\mathbb{P}}

\newcommand{\PP}{\mathbb{P}}

\DeclareMathOperator{\ch}{ch}

\DeclareMathOperator{\Hom}{Hom}

\DeclareMathOperator{\im}{im}
\DeclareMathOperator{\Sym}{Sym}

\DeclareMathOperator{\rk}{rk}

\DeclareMathOperator{\Ext}{Ext}

\DeclareMathOperator{\ev}{ev}

\DeclareMathOperator{\sHom}{\mathcal{H}\kern -.5pt\mathit{om}}
\DeclareMathOperator{\sTor}{\mathcal{T}\kern -1.5pt\mathit{or}}

\begin{document}

\title{Stability and cohomology of kernel bundles on $\mathbb{P}^n$}

\date{\today}
\author[I. Coskun]{Izzet Coskun}
\address{Department of Mathematics, Statistics and CS \\University of Illinois at Chicago, Chicago, IL 60607}
\email{icoskun@uic.edu}

\author[J. Huizenga]{Jack Huizenga}
\address{Department of Mathematics, The Pennsylvania State University, University Park, PA 16802}
\email{huizenga@psu.edu}

\author[G.  Smith]{Geoffrey Smith}
\address{Department of Mathematics, Statistics and CS \\University of Illinois at Chicago, Chicago, IL 60607}
\email{geoff@uic.edu}

\thanks{During the preparation of this article the first author was partially supported by the NSF FRG grant DMS 1664296
and the second author was partially supported by NSF FRG grant DMS 1664303.}

\begin{abstract}
In this paper, we study the cohomology of vector bundles on $\PP^n$ defined as kernels or cokernels of general maps $V_1 \to V_2$, where the $V_i$ are direct sums of line bundles or certain exceptional bundles. We prove an asymptotic cohomology vanishing theorem. We  characterize the stability of general Steiner bundles on $\PP^n$.  We also give a criterion for a Steiner bundle to be ample.
\end{abstract}

\maketitle

\setcounter{tocdepth}{1}
\tableofcontents

\section{Introduction}
In this paper, we study the cohomology of a  bundle on $\PP^n$ defined as the kernel or cokernel of a general map of vector bundles $V_1 \to V_2$, where $V_i$ are direct sums of line bundles or more generally direct sums of certain exceptional bundles. Our motivation comes from the study of vector bundles on $\PP^2$.

 Moduli spaces of stable vector bundles $M_{\PP^2}({\bf v})$ on $\PP^2$ are well-studied. Dr\'ezet and Le Potier have classified the Chern characters of stable bundles on $\PP^2$ \cite{DLP85}. The general bundle in $M_{\PP^2}({\bf v})$ has at most one nonzero cohomology group \cite{GH98} and admits a Gaeta resolution (see \cite{AGJ, CH18, CH20, Gae}). The Gaeta resolution allows one to classify the moduli spaces where the general bundle is globally generated and deduce the unirationality of the moduli space. Similarly, the ample and effective cones of $M_{\PP^2}({\bf v})$ have been computed \cite{CH16, CHW17}. The computations of the effective cone requires a resolution of the general bundle in terms of a judiciously chosen strong exceptional collection. Furthermore, the same resolution allows one to compute the cohomology of the tensor product of two general stable bundles, at least if one of the bundles has sufficiently divisible Chern character \cite{CHK21}. Given the importance of these two types of resolution for bundles on $\PP^2$, it is natural to ask for the properties of bundles on $\PP^n$ defined by similar resolutions.

First, we study the extent to which cohomology vanishing on $\PP^2$ extends to $\PP^n$ with $n\geq 3$. In \S \ref{sec-SKSmaximalrank}, we introduce the notion of a {\em strongly Kronecker stable pair of bundles} and give a criterion for when  maps between direct sums of such bundles induce maximal rank maps on global sections. In the case of line bundles, we can summarize our theorem as follows.

\begin{theorem}\label{firstTheorem}
Let $s_1,\ldots,s_a,t_1,\ldots,t_b$ be nonnegative integers.  Let
$$V_1 = \bigoplus_{1\leq i\leq a} \OO_{\PP^n}(d_i)^{s_i} \quad \mbox{and} \quad V_2 = \bigoplus_{1\leq j\leq b}\OO_{\PP^n}(e_j)^{t_j}$$ with $0 \leq d_i < e_j$ for all $i, j$. Suppose  $\rk(V_1)+\rk(V_2)\geq N$, where $N$ is an integer depending on $n$ and the integers $d_i$ and $e_j$.
Then a general map $M: V_1 \to V_2$ induces a maximal rank map on global sections. If  $\rk(V_1)\geq \rk(V_2)+n,$ then $V=\ker M$ is a locally free sheaf with $$h^0(V)=\max(\chi(V), 0), \quad   h^1(V)=\max(-\chi(V),0) \quad \mbox{and}\quad h^i(V)=0 \ \mbox{for} \ i \geq 2.$$ 
\end{theorem}

\begin{remark}
\begin{enumerate}
\item Theorem \ref{lineBundleTechnical} will make  Theorem \ref{firstTheorem} more precise by giving $N$ explicitly as a function of $n$ and the $d_i$ and $e_j$.

\item   We will deduce  Theorem \ref{lineBundleTechnical} from the more general Theorem \ref{mainCohomology}, which proves that general maps of strongly Kronecker stable pairs of bundles induce maps of maximal rank on global sections provided that the ranks are sufficiently large. Consequently, we will obtain analogues of Theorem \ref{firstTheorem} with line bundles replaced by certain twists of the tangent $T \PP^n$ and cotangent $\Omega_{\PP^n}$ bundles of $\PP^n$  (see Corollary \ref{cor-tangent}) or more general exceptional bundles. 
\end{enumerate}
\end{remark}

Some assumptions are needed on the exponents in Theorem \ref{firstTheorem} as the next example demonstrates.

\begin{example}\label{ex-scaling}
Let $V$ be the kernel of a general map
$$0 \to V \to \OO_{\PP^3}(2)^4 \to \OO_{\PP^3}(4) \to 0.$$ Then $h^0(V)=6$ and $h^1(V)=1$ (see Proposition \ref{rankn}), hence the induced map on cohomology does not have maximal rank. 
\end{example}

Example \ref{ex-scaling} is fairly typical. Let  $m \geq 2$ be an integer. One can generate infinitely many examples of bundles $V$ of rank $n$ on $\PP^n$ given as the  kernel of a general map $$\OO_{\PP^n}(a)^{t+n} \to \OO_{\PP^n}(a+m)^{t}$$  with two nonzero cohomology groups (see Remark \ref{rem-infinite}). When $V$ is the kernel of a general matrix of linear forms, we could not find such examples. Consequently, after some terminology, we make the following conjecture. 

A \emph{Steiner bundle} is a bundle $V$ admitting a presentation
\begin{equation}\label{eq-steiner}
0 \to \OO_{\PP^n}(-1)^t\to \OO_{\PP^n}^{t+r}\to V\to 0.
\end{equation}
These bundles were first introduced by Dolgachev and Kapranov \cite{DK93} and have since been intensively studied.
A  bundle has \emph{natural cohomology} if all its twists have at most one nonvanishing cohomology group.  The study of such bundles on $\PP^n$  was codified by \cite{BS08, ES09}, and related to the study of resolutions of Cohen--Macaulay modules.
\begin{conjecture}\label{conj-steiner}
 The  general Steiner bundle on $\PP^n$ has natural cohomology.
\end{conjecture}

\begin{remark}
\begin{enumerate} \item For the sequence (\ref{eq-steiner}) to be exact and $V$ to be a vector bundle, we need $r \geq n$.  Hence, if  $r<n$, Conjecture \ref{conj-steiner}  does not apply. 
\item By Serre duality, the naturality of $V$ above is equivalent to the naturality of $V^*$, which fits into the exact sequence
\[
0\ra V^*\ra \OO^{t+r}_{\PP^n}\ra \OO_{\PP^n}(1)^t\ra 0
\]
We study Conjecture \ref{conj-steiner} in this equivalent form because the statement reduces to maps between spaces of global sections having maximal rank.
\end{enumerate}
\end{remark}
Conjecture \ref{conj-steiner} holds on $\PP^2$ by \cite{CH20, GH98}. In the case of $\PP^2$, the  defining resolution of a Steiner bundle is the Gaeta resolution for $V$. Hence, $V$  is a  general member of the stack of prioritary sheaves with Chern character $\ch(V)$. Since the general bundle in these stacks has at most one nonzero cohomology group, we conclude that $V$ has at most one nonzero cohomology group. The same argument applies to all the twists of $V$. We show in Lemma \ref{lem-basicbound} that $V$ has natural cohomology if there are two particular twists of $V$ with natural cohomology.  Therefore the general Steiner bundle on $\PP^2$ has natural cohomology.

On any $\PP^n$, Ellia and Hirschowitz have studied the twists $V^*(1)$ \cite{EH92}. They show that $V^*(1)$ is  globally generated, and hence has no $H^1$,  if $s\geq 4$ and $r\geq \frac{(n+1)s}{2}+1$. 

 For $n \geq 3$, we will offer the following evidence for Conjecture \ref{conj-steiner}.

\begin{theorem}
Let $$0 \to V \to \OO_{\PP^n}^{m s} \to \OO_{\PP^n}(1)^{mt} \to 0$$ define a general kernel bundle $V$.
\begin{enumerate}
\item If $m \gg 0$, then $V$ has natural cohomology (see Theorem \ref{thm-Steinerscale}).
\item If the rank of $V$ is divisible by $n$, then $V$ has natural cohomology (see Corollary \ref{cor-divisiblebyn}). 
\end{enumerate}
\end{theorem}

In \S \ref{sec-stability}, we show that a general Steiner bundle $V$ is slope stable whenever it is not destabilized by an exceptional bundle. Combined with the work of Brambilla \cite{Bra05, Bra08} and Huizenga \cite{Hui13} characterizing stable exceptional Steiner bundles, our theorem gives a complete classification of Chern characters of stable Steiner bundles on $\PP^n$. Set $$\psi_n = \frac{n-1 + \sqrt{n^2+2n-3}}{2}.$$

\begin{theorem}[Theorem \ref{400}]\label{303}
Let $n\geq 2$, $r\geq n$, $t> 0$ be integers. Let $M$ be a general matrix of linear forms and let $V$ be the vector bundle on $\PP^n$ be defined by the sequence
\begin{equation*}
0\ra \OO_{\PP^n}(-1)^t\xrightarrow{M} \OO_{\PP^n}^{r+t}\ra V\ra 0.
\end{equation*}
If $\frac{n}{t} \leq \frac{r}{t} < \psi_n$, then $V$ is slope stable.
\end{theorem}
Stability of kernel and cokernel bundles has been an area of intense study. Previously,  Bohnhorst and Spindler \cite{BS92} gave precise criteria under which the rank $n$ cokernel of an injective map of direct sums of line bundles  on $\PP^n$ is stable. Cascini \cite{Cas02} proves Theorem \ref{303} in the specific case $r=n$ and $t=2$.
Coand\u{a} \cite{Coa11} and Costa, Marques, Mir\'{o}-Roig \cite{CMM10, MM11} have characterized the stability of syzygy bundles defined as the kernel of a map $\OO_{\PP^n}(-d)^s \ra \OO_{\PP^n}$. Our results imply the semistability  of certain kernels of general maps $\OO_{\PP^n}(-d)^s \to \OO_{\PP^n}^t$ (see Corollary \ref{cor-semistabilityd}).

Finally, we prove the following result about ampleness of Steiner bundles.
\begin{proposition}[Proposition \ref{prop-ample}]\label{314}
Let $V$ be a general Steiner bundle given by the presentation (\ref{eq-steiner}). If $t-r>2n-3$, then $V$ is ample.
\end{proposition}

\subsection*{Organization of the paper} In \S \ref{sec-SKSmaximalrank}, we introduce the notion of strongly Kronecker stable pairs of bundles and show that line bundles $\OO_{\PP^n}(i), \OO_{\PP^n}(j)$ for $0\leq i < j$ form a strongly Kronecker stable pair. We show that strong Kronecker stability implies asymptotic maximal rank and prove Conjecture \ref{conj-steiner} asymptotically. In \S \ref{sec-SKSmutation}, we show how to generate strongly Kronecker stable pairs via mutation, increasing the applicability of maximal rank statements in \S \ref{sec-SKSmaximalrank} greatly. In \S \ref{sec-rankn}, we study rank $n$ Steiner bundles and prove Conjecture \ref{conj-steiner} for Steiner bundles of rank divisible by $n$. In \S \ref{sec-stability} we completely settle the slope stability of Steiner bundles. In \S \ref{sec-ample}, we study ampleness for Steiner bundles and related cokernel bundles and prove Proposition \ref{prop-ample}.

\subsection*{Acknowledgments} We would like to thank Benjamin Gould, Yeqin Liu, and John Kopper for invaluable discussions.

\section{Strong Kronecker stability and maximal rank}\label{sec-SKSmaximalrank}
In this section, we introduce the notion of strongly Kronecker stable pairs of bundles and prove a criterion for a general map between two direct sums of line bundles to have maximal rank on cohomology. Our criterion applies more generally to maps between direct sums of vector bundles that are strongly Kronecker stable. We will generate other strongly Kronecker stable pairs in the next section.

\begin{definition}
Let $E_1$ and $E_2$ be two coherent sheaves on $\PP^n$. We say that the pair $E_1, E_2$ is {\em strongly Kronecker stable} and write  $E_1 < E_2$   if 
\begin{enumerate}
\item  $H^0(E_2)\neq 0$, 
\item the evaluation map $\ev: H^0(E_1)\otimes \Hom(E_1,E_2) \ra H^0(E_2)$ is surjective, and 
\item for any subspace $0\neq V\subsetneq H^0(E_1)$,  $$\frac{\dim(\ev(V\otimes \Hom(E_1,E_2)))}{h^0(E_2)}> \frac{\dim(V) }{h^0(E_1)}.$$  
\end{enumerate}
\end{definition}

\begin{remark}
Let $N= \hom(E_1, E_2)$ and let $K_{N}$ be the Kronecker quiver with two vertices and $N$ arrows. Condition (c) implies that the Kronecker module that associates the vector spaces $H^0(E_1)$ and $H^0(E_2)$ to the vertices of $K_{N}$ and the canonical evaluation to the arrows is a stable Kronecker module.
\end{remark}

\subsection{Kronecker stability of line bundles}
Our first result in this section shows that pairs of line bundles on $\PP^n$ are strongly Kronecker stable. 

\begin{theorem}\label{prop-linebundle}
For any integers $i \geq 0$ and $j >0$, $\OO_{\PP^n}(i)<\OO_{\PP^n}(i+j)$.
\end{theorem}
\begin{proof}
Since $i+j > 0$, $H^0(\OO_{\PP^n}(i+j)) \neq 0$ and the evaluation map $$\ev: H^0(\OO_{\PP^n}(i)) \otimes \Hom(\OO_{\PP^n}(i), \OO_{\PP^n}(i+j)) \ra H^0(\OO_{\PP^n}(i+j))$$ given by multiplication of polynomials of degree $i$ with polynomials of degree $j$ is  surjective. Hence, Theorem \ref{prop-linebundle} is an immediate consequence of the following lemma. \end{proof}

\begin{lemma}\label{lemma}
Set $V_i=H^0(\PP^n,\OO(i))$, and let $U\subset V_i$ be a vector subspace. Let $U_j$ be the image of the multiplication map $U\otimes V_j\ra V_{i+j}$. Then, if $0\subsetneq U\subsetneq V_i$, 
\[
\mathrm{dim}(U_j)> \dim(U)\frac{h^0(\OO_{\PP^n}(i+j))}{h^0(\OO_{\PP^n}(i))}.
\]
\end{lemma}
Lemma \ref{lemma} in turn follows from a theorem Green attributes to Macaulay \cite{Gre98}.
\begin{theorem}[Macaulay's bound on Hilbert functions]
Let $S$ be the homogeneous coordinate ring of $\PP^n$ and $I$ a homogeneous ideal of $S$. Suppose $I_d$ has codimension $c$ in $S_d$, where $c$ is a positive integer. Let $k_d>k_{d-1}>\cdots >k_\delta\geq\delta\geq 1$  be the uniquely determined integers such that
\[
c=\binom{k_d}{d}+\binom{k_{d-1}}{d-1}+\cdots +\binom{k_\delta}{\delta}.
\]
Set 
\[
c^{\langle d \rangle}=\binom{k_d+1}{d+1}+\binom{k_{d-1}+1}{d}+\cdots +\binom{k_\delta+1}{\delta+1}
\]
Then $I_{d+1}$ has codimension at most $c^{\langle d\rangle}$ in $S_{d+1}$.
\end{theorem}
\begin{proof}[Proof of Lemma \ref{lemma}]
We claim that for all $k\geq d$ and $\binom{k}{d}> c\geq 1$, we have
\begin{equation}\label{indHyp}
c^{\langle d\rangle} < \frac{k+1}{d+1}c.
\end{equation}
 We prove this by induction on $d$. If $d=1$, we have $c^{\langle 1 \rangle} = \frac{c(c+1)}{2} < \frac{k+1}{2} c$ for  $c<k=\binom{k}{1}$. Suppose that the bound holds for all $d'< d$. It suffices to prove the inequality for the integer $k$ such that $\binom{k-1}{d}\leq  c< \binom{k}{d}$ as increasing $k$ only improves the inequality. Then we have 
\[
c=\binom{k-1}{d}+c',
\]
where $c'$ is an integer satisfying $0\leq c'<\binom{k-1}{d-1}$. We then have
\[
c^{\langle d\rangle}=\binom{k}{d+1}+c'^{\langle d-1\rangle}.
\]
By the induction hypothesis, we then have
\[
c^{\langle d\rangle}<\binom{k}{d+1}+ \frac{k}{d} c'.
\]
Since $c'<\binom{k-1}{d-1}$ rearranges to $(k-d)c'<d\binom{k-1}{d}$, we conclude that
$$c^{\langle d\rangle}<\binom{k-1}{d}\frac{k}{d+1}+ \frac{k}{d}c' <\frac{k+1}{d+1}c +\frac{k-d}{d(d+1)}c'-\frac{1}{d+1}\binom{k-1}{d}< \frac{k+1}{d+1}c$$
as desired.

To complete the proof, let $U\subset V_i$  have codimension $c$ distinct from $0$ or $\mathrm{dim}(V_i)$.  Then $U_j\subset V_{i+j}$ has codimension at most $(\cdots ((c^{\langle i\rangle})^{\langle i+1 \rangle})\cdots)^{\langle i+j-1\rangle}$. Repeatedly applying (\ref{indHyp}) gives
\[
\mathrm{codim}(U_j)<\mathrm{codim}(U)\frac{n+i+1}{i+1}\cdots \frac{n+i+j}{i+j} = \mathrm{codim}(U)\frac{\binom{i+j+n}{i+j}}{\binom{i+n}{i}}.
\]
Hence, $$\dim(U_j) > \dim(U) \frac{h^0(\OO_{\PP^n}(i+j))}{h^0(\OO_{\PP^n}(i))},$$ as desired.
\end{proof}

\subsection{Kronecker stability and maximal rank}
The significance of the relation $E<F$ arises from the fact that it translates to a maximal rank statement on maps of global sections. Our main  theorem in this section is the following.

\begin{theorem}\label{mainCohomology}
Let $E_1,\ldots, E_a$ and $F_1,\ldots, F_b$ be coherent sheaves on $\PP^n$. Assume $E_i < F_j$ for all $i,j$.  Let $s_1,\ldots,s_a,t_1,\ldots,t_b$ be nonnegative integers, and let $ V_1=\bigoplus_{i=1}^a E_i^{s_i}$ and  $V_2= \bigoplus_{j=1}^b F_j^{t_j}$.  Assume that $$\max \{ h^0(V_1) , h^0(V_2)\}  \geq \frac{1}{4}\left(\max_{i} h^0(E_i)\right)\left(\max_{j} h^0(F_j)\right)\sum_{i=1}^a h^0(E_i)^2.$$  Let $M \in \Hom(V_1, V_2)$ be a general map. Then the induced map $M: H^0(V_1) \ra H^0(V_2)$ has maximal rank.
\end{theorem}
Theorems \ref{prop-linebundle} and \ref{mainCohomology} imply the following  theorem in the case of line bundles. 
 
 \begin{theorem}\label{lineBundleTechnical}
 Set $V_1 = \bigoplus_{i=1}^a  \OO_{\P^n}(d_i)^{s_i}$ and $V_2 = \bigoplus_{j=1}^b \OO_{\P^n}(e_j)^{t_j}$ be direct sums of line bundles on $\P^n$.  Assume $0\leq d_1 < \cdots <d_a < e_1 <\cdots <e_b$.
If 
$$\max\left\{ \sum_{i=1}^a s_i {d_i +n \choose n},\sum_{j=1}^b t_j{e_j+n\choose n}\right\} \geq \frac{1}{4}{d_a+n\choose n}{e_b+n\choose n}\sum_{i=1}^a {d_i+n\choose n}^2$$
then a general map $M:V_1\ra V_2$ induces a maximal rank map $M:H^0(V_1)\ra H^0(V_2)$.
 \end{theorem}

In particular, we obtain the following immediate corollaries.

\begin{corollary}\label{corOneLineBundle}
Let $0\leq d<e$ be two integers and let $M: \OO_{\PP^n}(d)^s \to \OO_{\PP^n}(e)^t$ be a general map. If $$\max\left\{ s {d+n\choose n},t{e+n\choose n}\right\} \geq \frac{1}{4} {d+n\choose n}^3{e+n\choose n}$$ then the induced map $M: H^0(\OO_{\PP^n}(d)^s) \to H^0( \OO_{\PP^n}(e)^t)$ has maximal rank. 
\end{corollary}

\begin{corollary}\label{corScalingLinebundle}
Let ${\bf v}$ be the Chern character of a sheaf $V$ on $\PP^n$ defined as the kernel of a general map $M$
$$0 \to V \to \bigoplus_{i=1}^a \OO(d_i)^{s_i}  \stackrel{M}{\to}  \bigoplus_{j=1}^b \OO(e_i)^{t_j} \to 0$$ with $-n\leq d_1 < \cdots < d_a < e_1 < \cdots < e_b$.
Then for $m$ sufficiently large, there exist vector bundles $V$ on $\PP^n$ with  Chern character $m {\bf v}$ such that $V$ has at most one nonzero cohomology group. If $\chi(V) \geq 0$, then $h^0(V) = \chi(V)$. If $\chi(V) < 0$, then $h^1(V) = - \chi(V)$.
\end{corollary}

\begin{proof}[Proof of Theorem \ref{mainCohomology}]
To fix notation, we write $$e_i = h^0(E_i), \quad f_j = h^0(F_j), \quad \textrm{and} \quad r_{ij} = \dim \Hom(E_i,F_j).$$  Therefore $$h^0(V_1) = \sum_i s_ie_i \quad \textrm{and} \quad h^0(V_2) = \sum_j t_jf_j.$$ Consider the incidence correspondence $$\Phi := \{(v,M)|Mv = 0\} \subset H^0(V_1)\times \Hom(V_1,V_2)$$ with projections \begin{align*}\alpha&:\Phi\to H^0(V_1)\\\beta&:\Phi\to \Hom(V_1,V_2).\end{align*}  Let $k= \max\{h^0(V_1)-h^0(V_2),0\}$, so that $k$ is the dimension of the kernel of $H^0(V_1)\fto M H^0(V_2)$ if this map has maximal rank.   The fibers of $\beta$ all have dimension at least $k$, so $$\dim \Phi \geq \hom(V_1,V_2)+k.$$  If we have an equality $\dim \Phi = \hom(V_1,V_2) + k$, then the general fiber of $\beta$ has dimension exactly $k$ and the general map $H^0(V_1)\fto M H^0(V_2)$ has maximal rank.  Thus, it suffices to prove $$\dim \Phi \leq \hom(V_1,V_2)+k.$$  

We estimate the dimension of $\Phi$ by studying the first projection $\alpha$.  The main difficulty here is that for certain vectors $v\in H^0(V_1)$ the fiber $\alpha^{-1}(v)$ can be very large.   We will need to stratify the space $H^0(V_1)$ into various loci where the fibers of $\alpha$ jump and control the dimensions of these loci and the fibers over them.

Since $H^0(V_1) = \bigoplus_i H^0(E_i)^{s_i}$, we can view a vector $v\in H^0(V_1)$ as a list of elements coming from the various subspaces $H^0(E_i)$.  Given a $v\in H^0(V_1)$, we construct the subspace $U_i(v) \subseteq H^0(E_i)$ which is spanned by the $s_i$ entries of $v$ which come from $H^0(E_i)$.  We let $\rho_i(v) = \dim U_i(v)$, and we associate to $v$ the vector of dimensions $$\vec\rho(v) := (\rho_1(v),\ldots,\rho_a(v)).$$  Notice that $0\leq \rho_i(v) \leq \min \{e_i,s_i\}$.  For any vector $\vec\rho = (\rho_1,\ldots,\rho_a)$  with $0\leq \rho_i \leq \min\{e_i,s_i\}$, we let $$S_{\vec\rho} = \{ v\in H^0(V_1) : \vec\rho(v) = \vec \rho\} \subset H^0(V_1)$$ and we let $$\Phi_{\vec\rho} = \alpha^{-1}(S_{\vec\rho}).$$  Then $H^0(V_1)$ is stratified by the various $S_{\vec\rho}$, so $\Phi$ is stratified by the various $\Phi_{\vec\rho}$.  There are only finitely many possible $\vec\rho$, so we need to show that $$\dim \Phi_{\vec\rho} \leq \hom(V_1,V_2) + k$$ for all $\vec\rho$.  Fix a single vector $\vec \rho = (\rho_1,\ldots,\rho_a)$ for the rest of the proof.

The dimension of $S_{\vec \rho}$ is easy to compute.  Constructing a vector $v\in S_{\vec \rho}$ amounts to selecting, for each $i$, a subspace $U_i \subset H^0(E_i)$ of dimension $\rho_i$, and a list of $s_i$ vectors from this subspace.  Therefore $$\dim S_{\vec \rho} = \sum_i \dim G(\rho_i,e_i) + \sum_i s_i\rho_i.$$

Next we estimate the dimension of the fibers of $\alpha : \Phi_{\vec \rho} \to S_{\vec \rho}.$    Viewing $M: \bigoplus_i E_i^{s_i} \to \bigoplus_j F_j^{t_j}$ as a matrix where each entry comes from some space $\Hom(E_i,F_j)$, each row of $M$ is an element of $$\Hom(E_1,F_j)^{s_1} \oplus \cdots \oplus \Hom(E_a,F_j)^{s_a}$$ for an appropriate $j$.  To have $Mv = 0$, each row of $M$ must come from the kernel of a corresponding evaluation map $$\psi_j(v):\Hom(E_1,F_j)^{s_1}\oplus\cdots \oplus \Hom(E_a,F_j)^{s_a} \to  H^0(F_j).$$ Let $\sigma_j(v) = \dim \im \psi_j(v)$ be the rank of this evaluation map.  Then $$\dim \ker \psi_j(v) = \sum_i s_i r_{ij} -\sigma_j(v),$$ and taking all the rows of $M$ into consideration, $$\dim \alpha^{-1}(v) = \sum_j t_j \dim \ker \psi_j(v) = \hom(V_1,V_2) - \sum_j t_j \sigma_j(v).$$ We define a vector $\vec \sigma(v) = (\sigma_1(v),\ldots,\sigma_b(v))$.  Given a vector $\vec \sigma = (\sigma_1,\ldots,\sigma_b)$, we further stratify the $S_{\vec \rho}$ into subsets $$S_{\vec \rho,\vec \sigma} = \{ v\in S_{\vec \rho}: \vec \sigma(v) = \vec \sigma\}.$$  We let $$\Phi_{\vec \rho,\vec \sigma} = \alpha^{-1} (S_{\vec \rho,\vec \sigma}),$$ which further stratifies $\Phi_{\vec \rho}$.  Thus it suffices to show each $\Phi_{\vec \rho,\vec \sigma}$ has dimension bounded by $\hom(V_1,V_2)+k$.  The dimension of $\alpha^{-1}(v)$ is constant for $v\in S_{\vec \rho,\vec \sigma}$, so since $S_{\vec \rho,\vec \sigma} \subset S_{\vec \rho}$ we have \begin{equation}\label{phiDim} \dim \Phi_{\vec \rho,\vec \sigma}= \dim S_{\vec\rho,\vec\sigma}+\dim \alpha^{-1}(v)\leq \sum_i \dim G(\rho_i,e_i) + \hom(V_1,V_2) + \sum_i s_i \rho_i - \sum_j t_j \sigma_j.\end{equation} For the rest of the proof we fix a vector $\vec\sigma$ such that $S_{\vec \rho,\vec \sigma}$ is nonempty, and write $v$ for an element of $S_{\vec \rho,\vec \sigma}$.

If for some $i$ we have  $\rho_i=e_i$, then  $U_i(v) = H^0(E_i)$.  Since $E_i < F_j$ for every $j$, the evaluation maps $\psi_j(v)$ are all surjective and $\sigma_j = f_j$.  Estimating $\dim S_{\vec \rho,\vec \sigma} \leq h^0(V_1)$,  we find $$\dim \Phi_{\vec \rho,\vec\sigma } = \dim S_{\vec \rho,\vec \sigma} + \dim \alpha^{-1}(v) \leq h^0(V_1)+\hom(V_1,V_2)-h^0(V_2) \leq \hom(V_1,V_2)+k.$$ Thus the proof is complete in this case.  In what follows we assume that $\rho_i < e_i$ for all $i$.

Let $1\leq m\leq a$ be the index such that $\rho_m/e_m$ is as large as possible and let $1\leq n \leq b$ be the index such that $\sigma_n/f_n$ is as small as possible.  Observe that the image of the evaluation map $$\phi_{mn}(v):U_m(v) \te \Hom(E_m,F_n)\to H^0(F_n)$$ is contained in $\im \psi_n(v)$.  Since $E_m < F_n$ and $\rho_m<e_m$, we have 
$$\frac{\rho_m}{e_m} < \frac{\dim \im \phi_{mn}(v)}{f_n} \leq \frac{\sigma_n}{f_n},$$
and therefore $$\frac{\rho_m}{e_m}\leq \frac{\sigma_n}{f_n} - \frac{1}{e_mf_n}.$$
Since $\rho_i/e_i \leq \rho_m/e_m$ we estimate $$\sum_i s_i\rho_i \leq \frac{\rho_m}{e_m}\sum_i s_i e_i = \frac{\rho_m}{e_m} h^0(V_1),$$ and similarly from $\sigma_i/f_i \geq \sigma_n/f_n$ we get $$\sum_j t_j\sigma_j \geq \frac{\sigma_n}{f_n} \sum_j t_jf_j = \frac{\sigma_n}{f_n} h^0(V_2).$$
Finally by estimating either $\rho_m/e_m$ or $\sigma_n/f_n$ we get the two inequalities $$\sum_i s_i \rho_i - \sum_j t_j \sigma_j \leq \frac{\rho_m}{e_m} h^0(V_1) - \frac{\sigma_n}{f_n}h^0(V_2)\leq \frac{\sigma_n}{f_n}(h^0(V_1)-h^0(V_2))-\frac{h^0(V_1)}{e_mf_n} \leq k - \frac{h^0(V_1)}{e_mf_n}$$ and $$\sum_i s_i\rho_i - \sum_j t_j \sigma_j \leq \frac{\rho_m}{e_m} h^0(V_1) - \frac{\sigma_n}{f_n} h^0(V_2) \leq \frac{\rho_m}{e_m} (h^0(V_1)-h^0(V_2))-\frac{h^0(V_2)}{e_mf_m}\leq k - \frac{h^0(V_2)}{e_mf_n},$$ which together give $$\sum_i s_i\rho_i - \sum_j t_j\sigma_j \leq k - \frac{1}{e_mf_n}\max\{h^0(V_1),h^0(V_2)\}.$$ Combining this with Inequality (\ref{phiDim}), we see that to prove $\dim \Phi_{\vec \rho,\vec \sigma} \leq \hom(V_1,V_2)+k$ it is enough to have $$\max\{h^0(V_1),h^0(V_2)\}\geq e_mf_n\sum_i \dim G(\rho_i,e_i) .$$ Here the dimension of the Grassmannian $G(\rho_i,e_i)$ is bounded by $e_i^2/4$, and the result follows.  
\end{proof}

We close this section by showing that Conjecture \ref{conj-steiner} holds asymptotically. We start with an elementary lemma.

\begin{lemma}\label{lem-basicbound}
Set $\alpha = \left\lceil \frac{tn}{r} \right\rceil $ and $\beta=\left \lfloor \frac{tn}{r}\right \rfloor$. A vector bundle of rank $r$ on $\PP^n$ given by a presentation
$$0 \to V \to \OO_{\PP^n}^{t+r} \to \OO_{\PP^n}(1)^t \to 0$$
has natural cohomology if and only if $$H^0\left(V\left(\beta-1\right)\right)=0 \quad \mbox{and} \quad H^1\left(V\left(\alpha-1\right)\right)=0.$$
\end{lemma}
\begin{proof}
Tensor the defining sequence of $V$ by $\OO_{\PP^n}(a)$. By the long exact sequence of cohomology, if $a < -1$, then $H^i(V(a)) = 0$ for $i<n$. In fact, if $-n-1< a <-1$, then $V(a)$ has no cohomology. Moreover, $h^1(V(-1))= t$ and all other cohomology of $V(-1)$ vanishes. From now on we assume that $a \geq 0$. In this case, $H^i(V(a))=0$ for $i \geq 2$. 
Observe that 
$$\chi(V(a)) \leq 0 \ \mbox{if} \ 0 \leq a \leq \beta-1 \quad \mbox{and} \quad \chi(V(a)) \geq 0 \ \mbox{if} \ a \geq \alpha -1.$$
Furthermore, equality occurs in either case exactly when $r$ divides $tn$. When $\chi(V(a))\leq 0$, we need to show that $H^0(V(a))=0$. Since $H^0(V(a-1)) \subset H^0(V(a))$, it suffices to verify the case $a =  \beta-1$. When  $\chi(V(a))\geq 0$, we need to show that $H^1(V(a))=0$. If $H^1(V(a))=0$, then $V$ is $(a+1)$-regular in the sense of Castelnuovo and Mumford. It follows that $V$ is $b$-regular for $b\geq a+1$ by Mumford's Theorem \cite[Theorem 1.8.3]{Laz04}. Consequently, for all $b>a$, $H^1(V(b))=0$. Hence, it suffices to verify the vanishing for $a =  \alpha-1.$ Since the necessity of the two conditions is clear, this concludes the proof of the lemma.
\end{proof}

\begin{theorem}\label{thm-Steinerscale}
Let $V$ be a  bundle of rank $mr$ on $\PP^n$ given by a general presentation of the form
$$0 \to V \to \OO_{\PP^n}^{m(t+r)} \to \OO_{\PP^n}(1)^{mt} \to 0.$$ Set $\alpha = \left\lceil \frac{tn}{r} \right\rceil $ and $\beta=\left \lfloor \frac{tn}{r}\right \rfloor$.
If $$
m\geq \max \left\{\frac{1}{4(t+r)} {\alpha+n-1 \choose n}^2{\alpha+n \choose n}, \frac{1}{4t} {\beta + n  -1 \choose n}^3\right\},
$$
then $V$ has natural cohomology.
\end{theorem}
\begin{proof}
 By Corollary \ref{corOneLineBundle}, since
 \[
 mt\geq \frac{1}{4}\binom{\beta+n-1}{n}^3,
 \]
 we have $H^0(V(\beta-1))=0$, and
 since
 \[
 m(t+r)\geq \frac{1}{4}\binom{\alpha+n-1}{n}^2\binom{\alpha+n}{n},
 \]
 we have $H^1(V(\alpha-1))=0$.
 By Lemma \ref{lem-basicbound}, $V$ has natural cohomology.
\end{proof}

\section{Strong Kronecker stability and mutations}\label{sec-SKSmutation}
In this section, given a pair of sheaves that are strongly Kronecker stable, we will show how to generate new strongly Kronecker stable pairs. We begin by showing that strong Kronecker stability is transitive.

\begin{lemma}\label{lem-transitive}
Let $E_1, E_2, E_3$ be three coherent sheaves on $\PP^n$. If $E_1<E_2$ and $E_2<E_3$, then $E_1<E_3$. 
\end{lemma}
\begin{proof}
Observe that $H^0(E_3) \neq 0$ by the assumption that $E_2 < E_3$. For $i<j$, write  $$\ev_{i,j}:H^0(E_i) \otimes \Hom(E_i, E_j) \ra H^0(E_j).$$ By assumption, $\ev_{1,2}$ and $\ev_{2,3}$ are surjective, therefore $\ev_{1,3}$ is also surjective. 

Let $0\neq V_1 \subsetneq H^0(E_1)$ be a subspace.  For $i=2,3$, let $V_i\subseteq H^0(E_i)$ be the image of $V_1 \otimes \Hom(E_1, E_i)$ under $ev_{1,i}$.  Let $V'_3$ be the image of $V_2\otimes \Hom(E_2,E_3)$ in $H^0(E_3)$. Observe that  $V'_3\subseteq V_3$. Set $v_i=\dim(V_i)$ and $v'_3=\dim(V'_3)$.  Since $0\neq V_1\subset H^0(E_1)$ and $E_1<E_2$,  we have 
\[
v_2>v_1\frac{h^0(E_2)}{h^0(E_1)}.
\]
If $V_2=H^0(E_2)$, then $V_3=V'_3=H^0(E_3)$ and the inequality $v_3>v_1\frac{h^0(E_3)}{h^0(E_1)}$ follows from $v_1<h^0(E_1)$.  Otherwise, since $E_2<E_3$, we have
\[
v'_3>v_2\frac{h^0(E_3)}{h^0(E_2)}
\]
Since $v_3\geq v'_3$,  the two inequalities above combine to
\[
v_3>v_1\frac{h^0(E_3)}{h^0(E_1)}.\qedhere
\]
\end{proof}

Recall that a  coherent sheaf $E$ on $\PP^n$ is
\begin{enumerate}
\item  {\em simple} if $\Hom(E,E) \cong \CC$;
\item {\em rigid} if $\Ext^1(E,E)=0$;
\item {\em exceptional} if $E$ is simple and $\Ext^i(E,E)=0$ for all $i>0$.
\end{enumerate}
Let $E$ and $F$ be two coherent sheaves on $\PP^n$ such that the evaluation map $$\ev: E \otimes \Hom(E, F) \ra F$$ is surjective. Then we define the {\em left mutation} $L_{E}F$ as the kernel of the evaluation map
\begin{equation}\label{def-leftmutation}
0 \ra L_{E}F \ra E \otimes \Hom(E,F) \ra F \ra 0.
\end{equation}
If the coevaluation map $$\ev^*:E\to F\te \Hom(E,F)^*$$ is injective, we define the {\em right mutation} $R_{F}E$ as its cokernel
 \begin{equation}\label{def-rightmutation}
0 \ra E \ra F \otimes \Hom(E,F)^* \ra  R_{F}E \ra 0.
 \end{equation}
Starting with a pair of strongly Kronecker stable line bundles, we can generate many other pairs of strongly Kronecker stable  bundles via mutations.
The following proposition summarizes the process.

\begin{proposition}\label{mutationProp}
Let $E,F$ be two coherent sheaves on $\PP^n$ such that $$\Hom(F,E)=\Ext^1(F,E)=0.$$ Assume that $E < F$.

\begin{enumerate}
\item If $E$ is simple, the evaluation map $E\te \Hom(E,F)\to F$ is surjective,  and $H^0(L_{E}F)\neq 0$, then $L_E F < E$. 
 \item If $F$ is simple, the coevaluation map $E\to F\te \Hom(E,F)^*$ is injective, and $H^1(E)= 0$, then $F<R_FE$. 
 \end{enumerate}
\end{proposition}
\begin{proof}
To prove (a), apply $\Hom(-, E)$ to the sequence (\ref{def-leftmutation}) defining $L_EF$. Since $$\Hom(F,E)=\Ext^1(F,E)=0$$ and $E$ is simple, we obtain an isomorphism 
$$\Hom(L_EF,E) \cong \Hom(E\otimes\Hom(E,F),E) \cong \Hom(E,F)^*.$$
Set $r=\dim(\Hom(E,F))$. Let $0\neq U\subseteq H^0(L_EF)$. Let $V \subseteq H^0(E)$ be the image of $U\otimes \Hom(L_EF,E)$ and let $W\subseteq H^0(F)$ be the image of $V\otimes \Hom(E,F)$. Denote the dimension of a vector space by the corresponding  lowercase letter.  If $V=H^0(E)$,  then $\frac{u}{v} \leq  \frac{h^0(L_EF)}{h^0(E)}$ for any $U$,  and this inequality is strict unless $U=H^0(L_EF)$.  Otherwise, since $E<F$, we have 
\begin{equation}\label{111}
\frac{v}{w}<\frac{h^0(E)}{h^0(F)}.
\end{equation} Let $K\subset H^0(L_EF)$ be the kernel of the map $V\otimes \Hom(E,F)\ra W$. By construction, we have $U\subseteq K$.  Since $\dim(K) = rv-w$, (\ref{111}) implies that \[
\dim(K)< v\left(r-\frac{h^0(F)}{h^0(E)}\right)
\]
so since $U\subseteq K$ and $h^0(L_EF)=rh^0(E)-h_0(F)$ we have
\[
u<v\frac{h^0(L_EF)}{h^0(E)}.
\]
If $U=H^0(L_EF)$ this argument shows that $H^0(L_EF)\otimes \Hom(L_EF,E) \ra H^0(E)$ is surjective. We conclude that $L_E F < E$.

To  prove (b), apply $\Hom(F, -)$ to the sequence (\ref{def-rightmutation}) defining $R_FE$ to see that 
$$\Hom(F, R_FE) \cong \Hom(E,F)^*.$$ The long exact sequence of cohomology associated to (\ref{def-rightmutation}) and our assumption that $H^1(E)=0$ imply that $H^0(F) \otimes \Hom(F, R_FE) \to H^0(R_FE)$ is surjective. Let $0 \neq U \subsetneq H^0(F)$ be a subspace. Let $W \subset H^0(R_FE)$ be the image of $U$ under the canonical evaluation and let $K \subset H^0(E)$ be the kernel.  Consider $K \otimes \Hom(E,F) \to H^0(F)$ and let $V$ be the image, which is a subspace of $U$. Since $E <F$, $u>k \frac{h^0( F)}{h^0(E)}.$ Since $h^0(R_FE) = r h^0(F) - h^0(E)$, $$w= ru -k > u \left(r - \frac{h^0(E)}{h^0(F)} \right) = u \frac{h^0(R_FE)}{h^0(F)}.$$ Thus $F < R_FE$.  
\end{proof}

\begin{definition}
Define a {\em positively constructible strong exceptional collection} for $\PP^n$  to be a strong exceptional collection obtained from the standard strong exceptional collection $$(\OO_{\PP^n}(a), \OO_{\PP^n}(a+1), \dots, \OO_{\PP^n}(a+n)) \  \mbox{for} \ a \geq 0$$ via a sequence of mutations that always preserves $\OO_{\PP^n}(a)$ on the left. 
\end{definition}
By \cite[Theorem 9.2]{Bon89}, mutations in the sense of (\ref{def-leftmutation}) or (\ref{def-rightmutation}) can always be performed starting with this collection, and moreover given $E,F$ with $E$ appearing to the left of $F$ in this category, we have $E\otimes \Hom(E,F)\ra F$ is surjective and $E\ra \Hom(E,F)^*\otimes F$ is injective. As a consequence, we have the following corollary.

\begin{corollary}\label{410}
Let $(E,F)$ be a pair of exceptional bundles fitting into a positively constructible full exceptional collection $(\OO_{\PP^n}(a),\ldots, E, F, \ldots)$ on $\PP^n$.  Let $L_0=E$, $L_1=L_{E} F$, and $L_i=L_{L_{i-1}} L_{i-2}$ for $i\geq 2$. Similarly, let $R_0=F$, $R_1=R_{F}E$, and $R_i=R_{R_{i-1}} R_{i-2}$ for $i \geq 2$. If $E<F$, then  $L_{i}<L_{i-1}$ and $R_{i-1}<R_{i}$ for all $i \geq 1$.
\end{corollary}
\begin{proof}
We proceed by induction on $i$.  If $i=1$, we have that $H^0(L_E F)\not=0$ and $H^1(L_E F)=0$, since it appears to the right of $\OO_{\PP^n}(a)$ in a strong exceptional collection. Hence $L_1 <L_0$ by Proposition \ref{mutationProp}. And if $i\geq 2$, we have that $(\OO_{\PP^n}(a),\ldots, L_{i},L_{i-1},\ldots)$ is a strong exceptional collection, so $L_i$ has $H^0(L_i)\neq 0$ and $H^1(L_i)=0$; from the latter we deduce that $L_{i-1}\otimes \Hom(L_{i-1},L_{i-2})\ra L_{i-2}$ is surjective. In addition, $L_{i-1}<L_{i-2}$ by the induction hypothesis. Then $L_i<L_{i-1}$ by Proposition \ref{mutationProp}.

Since the $R_i$ are all part of some positively constructible strong exceptional collection, we have $H^j(R_i)=0$ for all $i>0$ and $j>0$.   Induction and Proposition \ref{mutationProp} then imply $R_{i-1}<R_i$ for each $i\geq 1$.
\end{proof}
\begin{example}\label{example-tangent}
The Euler sequence exhibits $T \PP^n (e)$ as $R_{\OO_{\PP^n}(e+1)} \OO_{\PP^n}(e)$. By Proposition \ref{mutationProp}, we conclude that $\OO_{\PP^n}(e+1) < T\PP^n(e+1)$. Combining Theorem \ref{prop-linebundle} and Lemma \ref{lem-transitive}, we deduce that  $\OO_{\PP^n}(d)<T\PP^n(e)$ for all $1\leq d\leq e+1$. Similarly, by the Euler sequence, $\Omega_{\PP^n}(d)$ is the left mutation $L_{\OO_{\PP^n}(d-1)}\OO_{\PP^n}(d)$. We conclude that $\Omega_{\PP^n}(d)<\OO_{\PP^n}(e)$ for all $1\leq d \leq e+1$. 
\end{example}

Theorem \ref{mainCohomology} applied to Example \ref{example-tangent} implies the following corollary.

\begin{corollary}\label{cor-tangent}
Let $1\leq d\leq e+1$ and let $V$ be  a vector bundle defined by a general presentation 
\[
0 \to V \ra \Omega_{\PP^n}(d)^s\xrightarrow{M} T\PP^n(e)^t \to 0
\]
If $s$ or $t$ is sufficiently large that the inequality
$$
\max\left\{sh^0(\Omega_{\PP^n}(d)),th^0(T\PP^n(e))\right\}\geq \frac{1}{4}h^0(\Omega_{\PP^n}(d))^3h^0(T\PP^n(e))
$$
is satisfied, then $h^0(V)= \max(0, \chi(V))$, $h^1(V)= \max(0, -\chi(V))$ and $h^i(V)=0$ for $i \geq 2$. 

\end{corollary}

One can continue this process by taking further mutations to show that kernels and cokernels of general maps between arbitrarily complicated vector bundles have expected cohomology. 

It would be interesting to have a characterization of when two bundles satisfy $E < F$. For example, we raise the following questions.

\begin{question}
If $E$ and $F$ are two exceptional bundles on $\PP^n$ contained in the same helix and satisfying $0 < \mu(E) < \mu(F)$,  is $E < F$? 
\end{question}

\begin{question}
Let $E,F$ be two vector bundles on $\PP^n$. Does there exist integers $e$ and $f$ such that $E(e) < F(f)$?
\end{question}

\section{Rank \texorpdfstring{$n$}{n} kernel bundles}\label{sec-rankn}
In this section, we compute the cohomology of certain rank $n$ kernel bundles and verify Conjecture \ref{conj-steiner} for bundles on $\PP^n$ whose ranks are a multiple of $n$. The main technical observation is the following.

\begin{proposition}\label{rankn}
Let $V$ be a rank $n$ vector bundle on $\PP^n$ defined as the kernel of a surjective map of vector bundles $f: V_1 \to V_2$.
Let $d=\deg(V_1)-\deg(V_2)$. Set $$U_i=\Sym^i(V_2^*)\otimes \Lambda^{n-i-1}(V_1^*) (d) \quad \mbox{for} \ 0\leq i\leq n-1.$$ Suppose $H^{i+1}(U_i)=0$ for $0 \leq i \leq n-2$. Then
\[
h^0(V)=\sum_{i=0}^{n-1} \sum_{j=0}^i(-1)^{i+j}h^j(U_i). 
\]
If additionally $H^j(V_1)=H^j(V_2)=0$ for $j \geq 1$, then $H^j(V)=0$ for $j \geq 2$ and 
\[
h^1(V)=\sum_{i=0}^{n-1}\sum_{j=i+1}^n (-1)^{i+j+1}h^j(U_i) 
\]
\end{proposition}
\begin{proof}
Since $V$ is a vector bundle, we have an isomorphism $V\cong \Lambda^{n-1}(V^*)(d)$, and consequently we have the presentation of $V$ given by \cite[\S 2]{Wey03}
\[
0\ra U_{n-1}\ra U_{n-2}\ra \cdots \ra U_0\ra V\ra 0.
\]
For $0\leq i\leq n-3$, define $W_i$ by the exact sequences
\[
0\ra W_0\ra U_0\ra V \ra 0
\]
and 
\[
0\ra W_i\ra U_i\ra W_{i-1}\ra 0.
\]
We claim
\begin{equation}\label{151}
\sum_{j=0}^{i+1} (-1)^jh^j(W_i)=\sum_{k=i+1}^{n-1}\sum_{j=0}^k (-1)^{i+j+k+1}h^j(U_k).
\end{equation}
The proof is by downward descent on $i$. Since $H^{n-1}(U_{n-2})=0$, the  exact sequence 
\[
0\ra U_{n-1}\ra U_{n-2}\ra W_{n-3} \ra 0,
\]
induces the long exact sequence 
\[
0\ra H^0(U_{n-1})\ra H^0(U_{n-2})\ra H^0(W_{n-3})\ra \cdots \ra H^{n-1}(U_{n-1})\ra 0
\]
and (\ref{151}) follows when $i=n-3$. To induct, assume (\ref{151}) is true for $i>0$. Since $H^{i+1}(U_i)=0$, the exact sequence
\[
0\ra W_i\ra U_i\ra W_{i-1}\ra 0,
\]
induces the long exact sequence
\[
0\ra H^0(W_i)\ra H^0(U_i)\ra H^0(W_{i-1})\ra\cdots \ra H^{i+1}(W_i)\ra 0,
\]
which in turn gives
\[
\sum_{j=0}^i (-1)^jh^j(W_{i-1})=\sum_{j=0}^i (-1)^jh^j(U_i)+\sum_{j=0}^{i+1} (-1)^{j+1}h^j(W_{i}).
\]
Applying the induction hypothesis to the right-hand side gives (\ref{151}).

Finally, from the exact sequence 
\[
0\ra W_0\ra U_0\ra V\ra 0
\]
we have the long exact sequence
\[
0\ra H^0(W_0)\ra H^0(U_0)\ra H^0(V)\ra H^1(W_0)\ra 0
\]
so 
\[
h^0(V)=h^0(U_0)-h^0(W_0)+h^1(W_0).
\]
Applying (\ref{151}) with $i=0$ gives the formula for $h^0(V)$ in the proposition.

If furthermore $H^j(V_1) = H^j(V_2)=0$ for $j \geq 1$, then $H^j(V)=0$ for $j \geq 2$. We then have
\[
\chi(V) = h^0(V)-h^1(V)=\sum_{i=0}^ {n-1}(-1)^i\chi(U_i).
\]
Substituting the expression for $h^0(V)$, we obtain the expression for $h^1(V)$.
\end{proof}
\begin{corollary}\label{cor-rankn}
Let $V$ be a rank $n$ bundle defined by one of the following sequences $$0 \to V \to \OO_{\PP^n}^{t+n} \to \OO_{\PP^n}(1)^t \to 0, \ \mbox{or}$$  $$0 \to \OO_{\PP^n}^{t} \to \OO_{\PP^n}(1)^{t+n} \to V   \to 0.$$ Then $V$ has natural cohomology.
\end{corollary}

\begin{proof}
The corollary for kernel and cokernel bundles are equivalent by Serre duality, so we prove the statement for kernel bundles. 
We need to show that for any integer $a$, the twist $V(a)$ has at most one nonzero cohomology group. By Lemma \ref{lem-basicbound}, it suffices to check that the cohomology vanishes when $a= t-1$.  Now we may apply Proposition \ref{rankn}. In our case, $d=an-t$ and the bundle $U_i=\Sym^i(V_2^*)\otimes \Lambda^{n-i-1}(V_1^*) (d)$  is a direct sum of copies of the line bundle $\OO_{\PP^n}(a-t-i)$. In particular, the nonzero cohomology of all the vector bundles $U_i$ are  in degree $0$ if $a-t \geq 0$ and in degree $n$ if $a-t\leq -2$. All the cohomology vanishes if $a-t=-1$. Since 
$$h^0(V(a))=\sum_{i=0}^{n-1} \sum_{j=0}^{i}(-1)^{i+j}h^j(U_i),$$
we conclude that $h^0(V((a))=0$ if $a-t \leq -1$. Similarly, since
$$h^1(V)=\sum_{i=0}^{n-1}\sum_{j=i+1}^n (-1)^{i+j+1}h^j(U_i),$$ we conclude that $h^1(V)=0$ if $a-t \geq -1$. This concludes the proof.
\end{proof}

\begin{corollary}\label{cor-divisiblebyn}
A general Steiner bundle---or dual of a Steiner bundle---on $\PP^n$ of rank $kn$ has natural cohomology. In particular, Conjecture \ref{conj-steiner} is true if the rank of $V$ is a multiple of $n$.
\end{corollary}

\begin{proof}
Let $V$ be a general  bundle of rank $kn$ defined by a sequence
$$0 \to V \to \OO_{\PP^n}^{t+kn} \to \OO_{\PP^n}(1)^t \to 0.$$ Using the division algorithm, write $t= uk + v$, where $0 \leq v < k$. 
By Lemma \ref{lem-basicbound}, we need to show that $H^0(V(a))=0$ for $a\leq u-1$ and $H^1(V(a))=0$ for $a \geq u$. 
Let $V_1$ be a  bundle defined by a sequence
$$0 \to V_1 \to \OO_{\PP^n}^{u+n} \to \OO_{\PP^n}(1)^u \to 0$$ and let $V_2$ be a  bundle defined by a sequence
 $$0 \to V_2 \to \OO_{\PP^n}^{u+1+n} \to \OO_{\PP^n}(1)^{u+1} \to 0.$$
Then, by Corollary \ref{cor-rankn}, $V_1$ and $V_2$ have natural cohomology. More precisely, $H^0(V_1(a))=0$ for $a \leq u-1$ and $H^1(V_1(a))=0$ for $a \geq u-1$. Similarly, $H^0(V_2(a))=0$ for $a \leq u$ and $H^1(V_2(a))=0$ for $a \geq u$.  By specializing the matrix defining $V$ to a block diagonal form, we can specialize $V$ to the bundle $V' \cong V_1^{\oplus k-v} \oplus V_2^{\oplus v}$. Then $H^0(V'(a))=0$ if $a \leq u-1$ and $H^1(V'(a))=0$ if $a \geq u$. By Lemma \ref{lem-basicbound}, $V'$ has natural cohomology. In particular, the general  bundle $V$ satisfies $H^0(V(u-1))=0$ and $H^1(V(u))=0$ as well by the semicontinuity of cohomology.  So all twists of $V$ have at most one nonvanishing cohomology group.

Applying Serre duality to calculate the cohomology of the dual bundles $V^*$, we have the general Steiner bundle of rank $kn$ has natural cohomology as well.
\end{proof}
\begin{remark}\label{rem-infinite}
A vector bundle of rank $n$ given by a general presentation of the form 
$$0 \to V \to \OO_{\PP^n}^{t + n} \to \OO_{\PP^n}(d)^t \to 0$$ only has natural cohomology if $d=1$ or if $n=2$ and $d=2$. Consider $V(a)$. In Proposition \ref{rankn}, $U_i$ is a direct sum of line bundles of degree $a-dt-im$. If $a -dt \geq 0$ and $a-dt - (n-1)d \leq -n-1$, then $V(a)$ has nonzero $H^0$ and $H^1$. We may take $a =dt$. If $(n-1)d \geq n+1$, both inequalities are satisfied. This inequality holds when $d>2$ or $d=2$ and $n>2$. Hence, we can find infinitely many examples like Example \ref{ex-scaling}. When $n=2$ and $d=2$, then $V$ does have natural cohomology because $V$ is a general point of the stack of prioritary sheaves on $\PP^2$ given by a Gaeta resolution \cite{CH18, CH20}.
\end{remark}

We remark that the same technique as in Proposition \ref{rankn} computes most of the cohomology of twists of Steiner bundles and their duals.

\begin{proposition}
Let $V$ be a vector bundle on $\PP^n$ defined by the sequence
\[
0\ra V\ra \OO_{\PP^n}^{t+r}\ra \OO_{\PP^n}(1)^t\ra 0
\]
with $n\leq r< 2n$.  Then  $H^1(V(d))=0$ if $d\geq t-1$.
\end{proposition}
\begin{proof}

As in the proof of Proposition \ref{rankn},  we exploit the isomorphism $V\cong \Lambda^{r-1}(V^*)(\mathrm{det}(V))$. For $0\leq i\leq r-1$, set 
\[
U_i=\Sym^i(\OO(-1)^t)\otimes \Lambda^{r-i-1}(\OO_{\PP^n}^{r+t})\otimes \OO_{\PP^n}(-t)
\]
so $U_i\cong \OO_{\PP^n}(-t-i)^{r_i}$
 where $r_i$ is given by
\[
r_i=\binom{t+i-1}{i}\binom{r+t}{r-i-1}.
\]
Then we have a long exact sequence
\[
0\ra U_{r-1}(d)\ra \cdots \ra U_0(d)\ra V(d)\ra 0.
\]
Define $W_i$ as the kernel of the map $U_i\ra U_{i-1}$ for $i\geq 1$ and the kernel of the map $U_0\ra V$ if $i=0$. Since $H^j(U_i(d))=0$ for any $j\notin \{0,n\}$, the same argument as in Proposition \ref{rankn} shows $H^1(V(d))\cong H^2(W_0(d))\cong \ldots \cong H^n(W_{n-2}(d))$.  If $d\geq  t-1$, then $H^n(U_{n-1}(d))=0$, so since $W_{n-2}$ is a quotient of $U_{n-1}$, we have $H^1(V(d))=0$ as well.  
\end{proof}
\begin{remark}
This proof uses that $V$ has homological dimension $n-1$ to conclude that the cohomology is natural. This approach is similar to the one used  in \cite{MS17} to describe the cohomology of vector bundles of homological dimension less than $n$.
\end{remark}
\begin{corollary}
Let $V$ be a general vector bundle given by presentation
\[
0\ra V\ra \OO_{\PP^n}^{r+t}\ra\OO_{\PP^n}(1)^t \ra 0
\]and suppose $r=nk+\delta$ with $0\leq \delta <n$.  Then $H^0(V(d))=0$ if $d\leq\frac{t}{k+1}-1$ and $H^1(V(d))=0$ if $d\geq \frac{t}{k}-1$.
\end{corollary}
\begin{proof}
Let $V'$ be a general bundle with presentation
\[
0\ra V'\ra \OO_{\PP^n}^{(k+1)n+t}\xrightarrow{M} \OO_{\PP^n}(1)^t\ra 0.
\]
By Corollary \ref{cor-divisiblebyn},  if we have $d\leq \frac{t-k-1}{k+1}$,  then $H^0(V'(d))=0$.  Let $V$ be a subbundle of $V'$ given by restricting $M$ to a general $\OO_{\PP^n}^{r+t}\subset \OO(1)_{\PP^n}^{(k+1)n+t}$.  Then we have an exact sequence
\[
0\ra V\ra V'\ra \OO_{\PP^n}^{(k+1)n-r}\ra 0
\]
and $H^0(V(d))=0$ for $d\leq \frac{t-k-1}{k+1}$.

Likewise,  let $V''$ be a general bundle with presentation 
\[
0\ra V''\ra \OO_{\PP^n}^{kn+t}\ra \OO_{\PP^n}(1)^{t}\ra 0
\]
and let $V$ be an extension of $\OO^{r-kn}$ by $V''$, so we have a sequence
\[
0\ra V''\ra V\ra \OO_{\PP^n}^{r-kn}\ra 0.
\]
We have $H^1(V''(d))=0$ if $d\geq \frac{t}{k}-1$ by Corollary \ref{cor-divisiblebyn}, so $H^1(V(d))=0$ as well.
\end{proof}

\section{Stability of Steiner bundles on \texorpdfstring{$\PP^n$}{P3}}\label{sec-stability}
In this section, we classify slope stable Steiner bundles on $\PP^n$. Previously,  building on work of Brambilla \cite{Bra05, Bra08}, Huizenga classified slopes for which there exists semistable Steiner bundles \cite[Theorem 1.4]{Hui13}.
Set $$\rho_n(x)= \frac{1}{n-1+ \frac{1}{1+x}} \quad \mbox{and} \quad  \phi_n = \lim_{i \to \infty} \rho_n^i(0)= \frac{-n+1+ \sqrt{n^2+2n-3}}{2(n-1)},$$ where $\rho_n^i (x) = \rho_n(x) \circ \rho_n^{i-1}(x)$ and $\rho_n^0(x)= x$.
Let $$\Phi_n := \{\alpha : \alpha > \phi_n\} \cup \{\rho_n^i(0): i \geq 0\} \subset \QQ.$$
Huizenga proves that there exists a semistable Steiner bundle with slope $\mu$ if and only if $\mu \in \Phi_n$. Furthermore, the semistable bundles with $\mu < \phi_n$ are the semi-exceptional Steiner bundles and they are stable if and only if the rank and the degree are relatively prime. Huizenga shows that given the Chern character of a Steiner bundle with slope in $\Phi_n$, there exists a sufficiently large multiple of the Chern character which contains slope semistable bundles.
In addition,  Bohnhorst and Spindler \cite{BS92} give a precise criterion under which a rank $n$ cokernel of a map of direct sums  of line bundles is stable.

Set $$\psi_n = \frac{n-1 + \sqrt{n^2+2n-3}}{2}.$$ Note that $\psi_n = \frac{1}{\phi_n}.$
The main theorem of this section is the following.

\begin{theorem}\label{400}
Let $n\geq 2$, $r\geq n$, $t> 0$ be integers. Let $M$ be a general matrix of linear forms and let $V$ be the vector bundle on $\PP^n$ be defined by the sequence
\begin{equation}\label{402}
0\ra \OO_{\PP^n}(-1)^t\xrightarrow{M} \OO_{\PP^n}^{r+t}\ra V\ra 0.
\end{equation}
If $\frac{n}{t} \leq \frac{r}{t} < \psi_n$, then $V$ is slope stable.
\end{theorem}
\begin{remark}
The bound $\frac{r}{t}<\psi_n$ is equivalent to requiring that $V$ is not semi-exceptional and corresponds to the requirement that the Kronecker module associated to $V$ is stable. If $\frac{r}{t}> \psi_n$, then it is always possible to find integers $r'\leq r,t'$ with $\frac{r'}{t'}\leq \frac{r}{t}$ such that if $V'$ is defined by the sequence
\[
0\ra \OO_{\PP^n}(-1)^{t'}\xrightarrow{M} \OO_{\PP^n}^{r'+t'}\ra V'\ra 0,
\]
then $\chi(V',V)>0$. The only stable $V$ that have such a $V'$ are the exceptional bundles. This is discussed more extensively in \cite[Section 5]{Dre87}. Hence, Theorem \ref{400} combined with the classification of semi-exceptional Steiner bundles gives a complete description of slope stable Steiner bundles.
\end{remark}

Since the pullback of a slope semistable bundle under a finite morphism is slope semistable and semistability is open in families, we obtain the following immediate corollary of Theorem \ref{400}.

\begin{corollary}\label{cor-semistabilityd}
Let $n\geq 2$, $r\geq n$, $t> 0$ be integers. Let $M$ be a general matrix of forms of degree $d$ and let $V$ be the vector bundle on $\PP^n$ be defined by the sequence
\begin{equation}\label{886}
0\ra \OO_{\PP^n}(-d)^t\xrightarrow{M} \OO_{\PP^n}^{r+t}\ra V\ra 0.
\end{equation}
If $\frac{n}{t} \leq \frac{r}{t} < \psi_n$, then $V$ is slope semistable.
\end{corollary}

\begin{proof}[Proof of Theorem \ref{400}]

For $n=2$, the theorem follows from the work of Dr\'ezet and Le Potier \cite{DLP85}.

Now assuming $n\geq 3$, we will prove the theorem by restricting the bundle $V$ to a rational normal surface scroll embedded in $\PP^n$. For $e\in \{0,1\}$, let $\mathbb{F}_e$ denote the Hirzebruch surface $\PP(\OO_{\PP^1} \oplus \OO_{\PP^1}(e))$, let $F$ be the class of a fiber of the map $F_e\ra \PP^1$, and let $E$ be the class of a section with $E^2=-e$.
Recall that $\FF_0 \cong \PP^1 \times \PP^1$ and $\FF_1$ is the blowup of $\PP^2$ at a point.  The Picard group of $\FF_e$ is generated by the classes of  $E$ and  $F$. For $m>e$, the linear system $|E + mF|$  embeds $\FF_e$  into $\PP^{2m-e+1}$ as a rational normal scroll of degree $2m-e$.

We first prove the theorem under the additional assumption $\frac{n+1}{2}<\frac{r}{t}$. Let $S\subset \PP^n$ be a general rational normal surface scroll of degree $n$. When $n=2k+1$ is odd, $S$ is the image of $\PP^1 \times \PP^1$ under the linear system $|E+kF|$. When $n=2k$ is even, $S$ is the image of $\FF_1$ under the linear system $|E+kF|$.
Let $E$ be a general sheaf on $\PP^n$ as in  (\ref{402}).

Suppose there is a subsheaf $F\ra E$ with $\mathrm{rk}(F)< \mathrm{rk}(E)$ and $\mu(F)\geq \mu(E)$.
Since  $S$  is general, we may assume that the map $F\vert_S\ra V\vert_S$ is also an injection and $F\vert_S$ is torsion-free. Since $c_1(F\vert_S)$ and $c_1(V\vert_S)$ are both multiples of the hyperplane class on $S$, we have that $\mu_H(F\vert_S)\geq \mu_H(V\vert_S)$ for \emph{any} ample $H$ on $S$.
Restricting the sequence (\ref{402}) to $S$, we obtain a resolution of $V\vert_S$
$$0\to \OO_S(-1)^t \to \OO_S^{r+t} \to V\vert_S \to 0.$$ The stability of such bundles on Hirzebruch surfaces has been studied in \cite{CH21}. By  \cite[Theorem 10.9]{CH21}, if 
$\frac{n+1}{2}<\frac{r}{t}< \psi_n$, 
then $V\vert_S$ is slope stable with respect to ample divisors contained in a nonempty open subset of the ample cone of $S$. This contradicts the assumption that $F\ra E$ slope destabilizes $E$, showing that in this range $V$ is slope stable on $\PP^n$. 

Observe that when $m \geq 3$, $\frac{m+2}{2} < \psi_m$. Consequently, if $2 < \frac{r}{t} < \psi_n$, then there exists an integer $m \geq 3$ such that $\frac{m+1}{2} < \frac{r}{t} < \psi_m.$ The restriction of $V$ to a general linear $\PP^m$ is slope stable. It follows that $V$ is slope stable on $\PP^n$. This proves Theorem \ref{400} under the  assumption that $2 < \frac{r}{t} < \psi_n$. 

Finally, if $\frac{r}{t}\leq 2$ then we restrict to a general linear $\P^3$, and Theorem \ref{stable3Space} below shows that $V$ is slope stable.
\end{proof}
\begin{theorem}\label{stable3Space}
Let $r\geq 3,t> 0$ be integers and let the vector bundle $V$ on $\PP^3$ be defined by a sequence
\[
0\ra \OO_{\PP^3}(-1)^t\xrightarrow{M} \OO_{\PP^3}^{r+t}\ra V\ra 0,
\]
where $M$ is general.  If $r< (1+\sqrt{3}) t$, then $V$ is slope stable.
\end{theorem}
\begin{proof}
Let $Q$ be a smooth quadric surface with hyperplane class $H$, and let $W$ be a vector bundle defined by the sequence of sheaves on $Q$,
\begin{equation}\label{eq-51}
0\ra \OO_Q(-H)^t\xrightarrow{M} \OO_Q^{r+t}\ra W\ra 0,
\end{equation}
where $M$ is general.  We claim that $W$ is slope-stable with respect to $H$. Observe that $W$ is locally free. Let $L$ be a line on $Q$. We first show that $W$ is $L$-prioritary, that is $\Ext^2(W, W(-L))=0$.  Tensor the defining sequence (\ref{eq-51}) by $\OO_Q(-L)$ and apply $\Hom(W, -)$. The corresponding long exact sequence in cohomology, shows that $\Ext^2(W,W(-L))$ is a quotient of $$\Ext^2(W, \OO_Q(-L)) \cong H^2(W^*(-L))^{r+t}.$$  Dualize the defining sequence (\ref{eq-51}). Since  $H^1(\OO_Q(H-L))=H^2(\OO_Q(-L))=0$, we conclude that $H^2(W^*(-L))=0$.  In fact, $W$ is a general prioritary sheaf with Chern character $\ch(W)$ \cite{CH18, CH21, Wal98}.  The stack of prioritary sheaves is irreducible, and the locus of slope-stable sheaves is a dense open substack if it is nonempty.  Thus we must show there are slope-stable sheaves of character $\bw = \ch W=(r,tH,-t)$.

By work of Abe \cite[Theorem 3.7]{Abe}, there is an explicit ``Dr\'ezet-Le Potier'' function $\delta:\QQ\to \QQ$ such that if an integral Chern character $\bw$ has total slope $\frac{c_1}{r} = \mu(\bw) H$ proportional to $H$, then there is a slope-stable sheaf of character $\bw$ provided $\Delta(\bw) \geq \delta(\mu(\bw))$.  We compute $$\mu(\bw)  = \frac{t}{r} \quad \textrm{and}\quad \Delta(\bw) = \frac{t^2}{r^2}+\frac{t}{r} = \mu(\bw)^2+\mu(\bw).$$ A straightforward computation shows that if $\mu(\bw) >1/(1+\sqrt{3})$, then $\Delta(\bw) \geq \delta(\mu(\bw))$.  More visually, in Figure \ref{fig1} we graph both the function $\Delta = \delta(\mu)$ and the curve $\Delta = \mu^2+\mu$ for $\mu > 1/(1+\sqrt{3})$, and observe that the latter curve  lies on or above the former.
\end{proof}

\begin{figure}[t] 
\begin{center}
\setlength{\unitlength}{1in}
\begin{picture}(5.012,3.196)
\put(0,0){\includegraphics[scale=.4,bb=0 0 12.53in 7.99in]{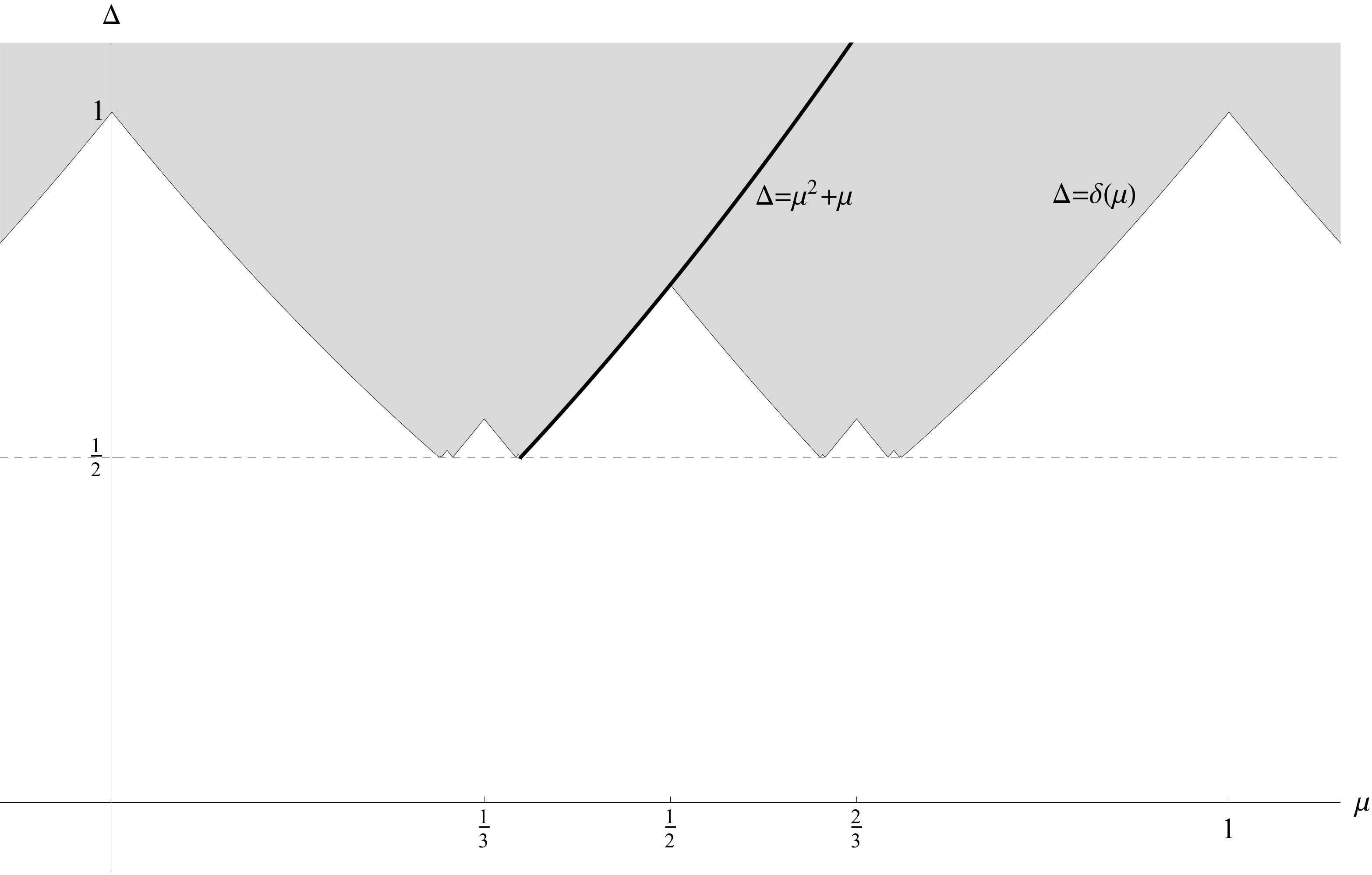}}
\end{picture}
\end{center}
\caption{The curve $\Delta = \delta(\mu)$ appears in Abe's classification of slope stable bundles with symmetric $c_1$ on a smooth quadric surface; slope stable bundles of slope $\mu$ and discriminant $\Delta$ exist if $(\mu,\Delta)$ lies on or above the curve, in the shaded region.  The slope and discriminant of the bundle $W$ in the proof of Theorem \ref{stable3Space} lie on the curve $\Delta = \mu^2+\mu$.  For $\mu > 1/(1+\sqrt{3}) \approx 0.36$, this curve lies on or above $\Delta = \delta(\mu)$.}\label{fig1}
\end{figure}
\begin{remark}
On $\PP^2$, exceptional bundles control the classification of stable bundles. Let $E$ be an exceptional bundle and let $V$ be a stable bundle with $\mu(E)-3< \mu(V) < \mu(E)$, then $\chi(E, V) \leq 0$. In fact, these types of inequalities determine the Chern characters of non-exceptional stable bundles on $\PP^2$ \cite{DLP85}. On $\PP^n$ consider  stable bundles of the form 
$$0 \to \OO_{\PP^n}(-n-1)^t \to \OO_{\PP^n}(-n)^{t+r} \to V \to 0.$$ By Theorem \ref{400}, by making the ratio $\frac{t}{r}$ less than $n$ but arbitrarily close to $n$, we can find stable bundles $V$ with slope less than but arbitrarily close to $0$ such that  $\chi(\OO, V) =  (-1)^{n-1}t$. In particular, if $n$ is odd, $\chi(\OO,V)>0$.
\end{remark}

\section{Ampleness of Steiner bundles}\label{sec-ample}
In this section, we give a criterion for  ampleness of  Steiner bundles.  

\begin{proposition}\label{prop-ample}
Let $t > 0$ and $r\geq n$ be positive integers. Let $V$ be a vector bundle on $\PP^n$ given by the presentation
\[
0\ra \OO_{\PP^n}(-1)^t\xrightarrow{M} \OO_{\PP^n}^{r+t}\ra V\ra 0.
\]
There is a closed subset $S\subset \Hom(\OO_{\PP^n}(-1)^t,\OO_{\PP^n}^r)$ of codimension at least $t-r-2n+3$ such that if $M\notin S$, the bundle $V$ is ample. In particular, if $t-r > 2n-3$, then the general $V$ is ample.
\end{proposition}
\begin{proof}
The bundle $V$ is globally generated since it is a quotient of $\OO^{t+r}$. Hence, by Gieseker's Lemma \cite[Proposition 6.1.7]{Laz04}, if $V$ is not ample, there is a curve $i:C\ra \PP^n$ such that $i^*V$ has a trivial quotient.  Hence, if $V$ is not ample, there is some $i:C\ra\PP^n$ such that $H^0(i^*(V^*))\neq 0$. Since $V$ is locally free, we have an exact sequence
\[
0\ra i^* (V^*)\ra \OO_C^{r+t}\xrightarrow{M^*} \OO_C(1)^{t}\ra 0.
\]
In particular,  $M^*:H^0(\OO^{r+t}_C)\ra H^0(\OO_C(1)^t)$ has nonzero kernel on $C$.  After a change of basis, we may assume that the vector $(1, 0, \dots, 0)$ is in the kernel. It follows that the first column of the matrix $M^*$ restricted to $C$ is identically zero.  Since $M^*$ is a matrix of linear forms on $\PP^n$, the first column of $M^*$ must be zero when restricted to the linear span of $C$. Hence, there exists a line $L$ such that $V^*$ restricted to $L$ has a section. We conclude that $V$ is ample if and only if all of its restrictions to lines in $\PP^n$ are ample.

Given a line $\ell\subset \PP^n$ and a nonzero section $s\in H^0(\ell, \OO_{\ell}^{r+t})$ the codimension in $\Hom(\OO_\ell^{r+t}, \OO_\ell(1)^t)$ of maps $M$ with $Ms=0$ is 2t. As a result, the set of $M:\OO_\ell^{r+t}\ra \OO_{\ell}(1)^t$ such that $M:H^0(\OO_\ell^{r+t})\ra H^0(\OO_\ell(1)^t)$ has nonzero kernel has codimension at least $t-r+1$ in $\Hom(\OO_\ell^{r+t}, \OO_\ell(1)^t)$.
Moreover, the restriction $\Hom_{\PP^n}(\OO_{\PP^n}^{r+t},\OO_{\PP^n}(1)^t)\ra \Hom_\ell(\OO_{\ell}^{r+t}, \OO_{\ell}(1)^t)$ is surjective, and as a consequence the codimension in $\Hom_{\PP^n}(\OO_{\PP^n}(-1)^t,\OO_{\PP^n}^{r+t})$ of the set of $M$ such that $V\vert_\ell$ fails to  be ample has codimension at least $t-r+1$ as well.  

Since the set of lines in $\PP^n$ has dimension $2n-2$, we have that the overall codimension of the set of $M$ such that $V$ is a vector bundle but fails to be ample on a line is at least $t-r-2n+3$. The result follows since any vector bundle that fails to be ample fails to be ample on a line.
\end{proof}

\begin{remark}
The expected codimension of the space of lines $\ell$ where the splitting type of $V|_{\ell}$  has a trivial factor is $t-r+1$. Since the Grassmannian $\mathbb{G}(1,n)$ has dimension $2n-2$, if $t-r \leq 2n-3$, one would not expect $V$ to be ample.
\end{remark}

\begin{corollary}
Let $t > 0$ and $r\geq n$ be positive integers. Assume that $t-r > 2n-3$. Then the general bundle $V$ having the presentation 
\[
0\ra \OO_{\PP^n}(-d)^t\xrightarrow{M} \OO_{\PP^n}^{r+s}\ra V\ra 0
\] is ample. 
\end{corollary}

\begin{proof}
The pullback of an ample vector bundle under a finite morphism is ample \cite[Proposition 6.1.2]{Laz04}. Consider a general  self map $f$ of $\PP^n$ of degree $d$. If we pullback an ample bundle $V$  in Proposition   \ref{prop-ample}, $f^* V$ is also ample. Since ampleness is an open condition, then for the general map $M \in \Hom(\OO_{\PP^n}(-d)^t, \OO_{\PP^n}^{t+r})$ is also ample.
\end{proof}

If $V$ is a Steiner bundle and $a>0$ is an integer, then $V(a)$ is ample since Steiner bundles are globally generated. On the other hand, if $a < 0$, then $V(a)$ does not have any sections. We conclude this section with the following question.

\begin{question}
Let $a <0$ be a negative integer. Let  $V$ be a general  bundle defined by 
$$0 \to \OO_{\PP^n}(a)^t  \to \OO_{\PP^n}(a+1)^{t+r} \to V \to 0.$$
If $t \gg 0$ (depending on $a$ and $r$), is $V$ ample? 
\end{question}


\begin{thebibliography}{10}

\bibitem{Abe}
Takeshi Abe, Semistable sheaves with symmetric $c_1$ on a quadric surface, Nagoya Math. J. {\bf 227} (2017), 86--159.


\bibitem{AGJ}
Aaron Bertram, Thomas Goller and Drew Johnson.
\newblock Le Potier’s strange duality, quot schemes and multiple point formulas for del Pezzo surfaces, preprint.


\bibitem{BS92}
Guntram Bohnhorst and Heinz Spindler.
\newblock The stability of certain vector bundles on $\PP^n$.
\newblock in {\em Complex Algebraic Varieties}, {\em Lecture Notes in Mathematics}, 1507:37--50, 1992.

\bibitem{BS08}
Mats Boij and jonas S{\"o}derberg.
\newblock Graded Betti numbers of Cohen--Macaulay modules and the multiplicity conjecture.
\newblock {\em J. London Math. Soc.}, 78(1): 85--106, 2008.

\bibitem{Bon89}
Alexei Igorevich Bondal.
\newblock Representations of associative algebras and coherent sheaves.
\newblock {\em Izv. Akad. Nauk. SSSR Ser. Mat.}53: 25--44, 1989.
Translation in {\em Math. USSR-Izv.} 34: 23--42, 1990.

\bibitem{Bra05}
Maria Chiara Brambilla.
\newblock Simplicity of generic Steiner bundles. 
\newblock {\em Boll. Unione Mat. Ital. Sez. B Artic. Ric. Mat. (8)}, 8(3):723--735, 2005.

\bibitem{Bra08}
Maria Chiara Brambilla.
\newblock Cokernel bundles and Fibonacci bundles. 
\newblock {\em Math. Nachr.} 281(4): 499--516, 2008.

\bibitem{Cas02}
Paolo Cascini, 
\newblock On the moduli space of the Schwarzenberger bundles. 
\newblock {\em Pacific J. Math.}, 205(2):311--323, 2002. 


\bibitem{Coa11}
Iustin Coand\u{a}.
\newblock On the stability of syzygy bundles. 
\newblock {\em Internat. J. Math.} 22(4):515--534, 2011.  

\bibitem{CMM10}
Laura Costa, Pedro Macias Marques, and Rosa Mar\'{i}a Mir\'{o}-Roig, 
\newblock Stability and un-obstructedness of syzygy bundles, 
\newblock {\em J. Pure  Appl. Algebra} 214(7):1241--1262,  2010.


\bibitem{CH16}
Izzet Coskun and Jack Huizenga.
\newblock The ample cone of moduli spaces of sheaves on the plane. 
\newblock {\em Algebr. Geom.} 3(1):106--136,  2016.

\bibitem{CH18}
Izzet Coskun and Jack Huizenga.
\newblock Weak Brill-{N}oether for rational surfaces.
\newblock in {\em Local and Global Methods in Algebraic Geometry}, {\em Contemp. Math.}, 712:81--104, 2018.



\bibitem{CH20}
Izzet Coskun and Jack Huizenga.
\newblock Brill-{N}oether theorems and globally generated vector bundles on
  {H}irzebruch surfaces.
\newblock {\em Nagoya Math. J.}, 238:1--36, 2020.

\bibitem{CH21}
Izzet Coskun and Jack Huizenga.
\newblock Existence of semistable sheaves on Hirzebruch surfaces.
\newblock {\em Adv.  Math.}, 381, 96 p., 2021.

\bibitem{CHK21}
Izzet Coskun, Jack Huizenga and John Kopper.
\newblock The cohomology of general tensor products of vector bundles on $\PP^2$. 
\newblock Selecta Math. (N.S.) 27(5), Paper No.94, 46 p.,  2021. 

\bibitem{CHW17}
Izzet Coskun, Jack Huizenga and Matthew Woolf.
\newblock The effective cone of the moduli space of sheaves on the plane.
\newblock {\em J. Eur. Math. Soc.}, 19(5):1421--1467, 2017.
 

\bibitem{DK93}
Igor Dolgachev and Mikhail Kapranov.
\newblock Arrangements of hyperplanes and vector bundles on $\PP^n$.
\newblock {\em Duke Math. J.}, 71(3): 633--664, 1993. 

\bibitem{Dre87}
Jean-Marie Dr\'ezet.
\newblock Fibr\'es exceptionnels et vari\'et\'es de modules de faisceaux semi-stables sur $\PP^2(\CC)$.
\newblock {\em J. Reine Angew. Math}, 380:14--58, 1987.

\bibitem{DLP85}
Jean-Marie Dr\'ezet and Joseph Le~Potier.
\newblock Fibr\'es stables et fibr\'es exceptionnels sur $\mathbb{P}^2$.
\newblock {\em Ann. Sci. \'{E}c. Norm. Sup\'er.},
  18(2):193--243, 1985.

\bibitem{ES09}
David Eisenbud and Frank-Olaf Schreyer.
\newblock Betti numbers of graded modules and cohomology of vector bundles.
\newblock {\em J. Amer. Math. Soc.} 22(3): 859--888, 2009.

\bibitem{EH92}
Philippe Ellia and Andr\'e Hirschowitz.
\newblock Voie ouest. I. G{\'e}n{\'e}ration de certains fibr{\'e}s sur les espaces projectifs et application.
\newblock {\em J. Algebraic Geometry}, 1(4):531--547, 1992.


\bibitem{Gae}
F. Gaeta.
 \newblock Sur la distribution des degr\'{e}s des formes appartenant \`{a} la matrice de l’id\'{e}al homog\'{e}ne attach\'{e} \`{a} un groupe de N points g\'{e}n\'{e}riques du plan, 
 \newblock {\em C. R. Acad. Sci. Paris} 233:912--913,  1951.


\bibitem{GH98}
Lothar G{\"o}ttsche and Andr{\'e} Hirschowitz.
\newblock Weak Brill-Noether for vector bundles on the projective plane.
\newblock In {\em Algebraic geometry (Catania, 1993/Barcelona, 1994)}, volume
  200 of {\em Lecture notes in Pure and Applied Mathematics}, pages 63--74.
  Marcel Dekker, 1998.
  
  

\bibitem{Gre98}
Mark~L Green.
\newblock Generic initial ideals.
\newblock In {\em Six lectures on commutative algebra}, pages 119--186.
  Springer, 1998.
  
  \bibitem{Hui13}
  Jack Huizenga, 
 \newblock Restrictions of Steiner bundles and divisors on the Hilbert scheme of points in the plane. 
\newblock {\em  Int. Math. Res. Not. IMRN} 21:4829--4873, 2013. 

\bibitem{Laz04}
Robert Lazarsfeld.
\newblock {\em Positivity in Algebraic Geometry I and II}.
\newblock Springer, 2004.

\bibitem{MM11}
Pedro Macias Marques and Rosa Mar\'{i}a Mir\'{o}-Roig.
\newblock Stability of syzygy bundles. 
\newblock {\em Proc.  Amer.  Math. Soc.},  139(9):3155--3170,  2011. 

\bibitem{MS17}
Rosa Mar\'{i}a Mir\'{o}-Roig and Helena Soares
\newblock Exceptional bundles of homological dimension $k$.
\newblock {\em Forum Mathematicum}, 29(3):701--715, 2017.

\bibitem{Rud94}
Alexei Rudakov.
\newblock A description of {C}hern classes of semistable sheaves on a quadric
  surface.
\newblock {\em J. Reine angew. Math.}, 453:113--135, 1994.

\bibitem{Wal98}
Charles Walter.
\newblock Irreducibility of moduli spaces of vector bundles on birationally
  ruled surfaces.
\newblock In {\em Algebraic geometry (Catania, 1993/Barcelona, 1994)}, volume
  200 of {\em Lecture notes in Pure and Applied Mathematics}, pages 201--211.
  Marcel Dekker, 1998.



\bibitem{Wey03}
Jerzy M. Weyman.
\newblock {\em Cohomology of vector bundles and syzygies}
\newblock Cambridge University Press, 2003.


\end{thebibliography}
\end{document}